\documentclass[11pt]{amsart}
\usepackage{indentfirst,latexsym,bm,color}
\usepackage{amsmath,amssymb}
\usepackage{amsthm}
\usepackage{fancyhdr}
\usepackage{graphics}
\usepackage{indentfirst,latexsym,bm,amsthm,amssymb,graphicx}
\usepackage{colortbl}
\usepackage{times}
\usepackage{amsfonts}

\newtheorem{theorem}{Theorem}[section]
\newtheorem{lemm}[theorem]{Lemma}

\newtheorem{theo}[theorem]{Theorem}
\theoremstyle{definition}

\newtheorem{coro}[theorem]{Corollary}
\theoremstyle{remark}
\newtheorem{remark}[theorem]{Remark}

\renewcommand{\thefootnote}

\begin{document}

\title[Poincar\'{e} series and generalized McKay-Slodowy quiver]
{Poincar\'{e} series of relative symmetric invariants 
for $\mathrm{SL}_n(\mathbb{C})$}

\author[Jing]{Naihuan Jing}
\address{
Department of Mathematics, North Carolina State University,
   Raleigh, NC 27695, USA \\}
\email{jing@math.ncsu.edu}

\author[Wang]{Danxia Wang}
\address{Department of Mathematics, Shanghai University,
Shanghai 200444, China}
\email{dxwangmath@126.com}

\author[Zhang]{Honglian Zhang}
\address{Department of Mathematics, Shanghai University,
Shanghai 200444, China} \email{hlzhangmath@shu.edu.cn}

\begin{abstract}
Let $(N, G)$, where $N\unlhd G\leq \mathrm{SL}_n(\mathbb{C})$, be a pair of finite groups and
$V$ a finite-dimensional fundamental $G$-module. We study the $G$-invariants in the symmetric algebra $S(V)=\oplus_{k \geq 0}S^k(V)$ by giving explicit  formulas of the Poincar\'{e} series
for the induced modules and restriction modules. In particular, this provides a uniform formula of the Poincar\'{e} series for the symmetric invariants in terms of the McKay-Slodowy correspondence. Moreover, we also derive a global version of the Poincar\'e series
in terms of Tchebychev polynomials in the sense that one needs only the dimensions of the
subgroups and their group-types to completely determine the Poincar\'e series.
\end{abstract}
\date{}
\maketitle

\footnote{Corresponding author: dxwangmath@126.com}
\section{Introduction}

Let $V=\mathbb C^n$ be the fundamental module of the special linear group $\mathrm{SL}_n(\mathbb C)$. The $k$th symmetric tensor space
$S^k(V)$ is a simple $\mathrm{SL}_n(\mathbb C)$-module denoted by $\Gamma_{(k)}$, and the symmetric tensor algebra $S(V)=\bigoplus_{k=0}^{\infty}S^k(V)$
is an infinite dimensional $\mathrm{SL}_n(\mathbb C)$-module. It is well-known that the set
$\{\Gamma_{(k)}\}$ generates all irreducible finite dimensional
representations of $\mathrm{SL}_n$ under the restriction functor according to the Schur-Weyl duality.

Let $G$ be a finite subgroup of $\mathrm{SL}_2(\mathbb C)$. Kostant \cite{Kos, Kos2, Kos3} studied the interesting question on how the restriction
$\Gamma_{(k)}|_{G}$ decomposes itself into simple $G$-modules and found that the answer relies upon resolution of singularity of certain algebraic surfaces through the McKay correspondence.

The McKay correspondence \cite{Mc} gives a bijective map between finite subgroups of $\mathrm{SL}_2(\mathbb{C})$ and affine Dynkin diagrams of untwisted $ADE$ types.
It is known that this correspondence establishes a classification of resolution of singularities of $\mathbb C^2/G$, where
$G$ is a finite subgroup of $\mathrm{SL}_2(\mathbb C)$.
Slodowy \cite{Sl} considered more general minimal resolution of the singularity of $\mathbb C^2/ {N}$ under the action of
$G/N$, where $G\leq \mathrm{SL}_2(\mathbb C)$. The algebraic counterpart is the so-called  McKay-Slodowy correspondence which matches all affine Dynkin diagrams
with the pairs $N\lhd G\leq \mathrm{SL}_2(\mathbb C)$. For an elementary proof of the McKay-Slodowy correspondence, see \cite{JWZ}. 

The goal of this paper is to generalize Kostant's results in two directions.  In the first direction, we consider the general special linear group
$\mathrm{SL}_n(\mathbb C)$. In another aspect, we replace the defining fundamental module $V=\mathbb C^n$ by any fundamental irreducible module of the Lie group $\mathrm{SL}_n(\mathbb C)$.
The main results of the paper will show that even in the most general situation of $\mathrm{SL}_n(\mathbb C)$, one can still obtain
similar formulas for the Poincar\'e series of $G$-invariants in the symmetric tensor algebra, and in the special case of the twisted affine Lie algebras, some of the intrinsic data are also encoded in the
Poincar\'e series of the $(G, N)$-invariants in the symmetric tensor algebra (compare \cite{JWZ}).


Let us now describe the main results in the paper. 
Let $\mathfrak{h}$ be a Cartan subalgebra of the special linear Lie algebra $\mathfrak{sl}_n(\mathbb{C})$,
and let $P$ be the weight lattice spanned by the vectors $\varepsilon_1,\varepsilon_2,\ldots,\varepsilon_{n-1}$, where $(\varepsilon_i, \varepsilon_j)=\delta_{ij}$.
Denote by $\Gamma_{(u_1,u_2,\ldots, u_{n-1})}$ the finite dimensional irreducible representation of the special linear group $\mathrm{SL}_n(\mathbb{C})$
associated to the dominant weight $\lambda=(u_1+ \ldots + u_{n-1})\varepsilon_1+(u_2+ \ldots + u_{n-1})\varepsilon_2+ \ldots + u_{n-1}\varepsilon_{n-1}$,
where
$(u_1,u_2, \ldots ,u_{n-1})\in \mathbb Z_+^{n-1}$. Therefore $\Gamma_{(1, 0, \ldots, 0)}$ is the standard module $V$, and
the $r$th fundamental module $\Gamma_{(0, \ldots, 1, \ldots, 0)}=\wedge^r V$.

Let $N$ be a normal subgroup of a finite group $G\leq \mathrm{SL}_n(\mathbb{C})$ and
$\{\rho_i|i\in {\rm I}_{G}\}$ (resp. $\{\phi_i|i\in {\rm I}_{N}\}$) the set of
complex finite-dimensional irreducible modules of $G$ (resp. $N$).
Let $\{\check\rho_i|i\in \check{\mathrm{I}}\}$ be the set of inequivalent $N$-restriction modules ${\rm Res}(\rho_i):=\check\rho_i$.
Correspondingly, the set $\{\hat\phi_i|i\in \rm\hat{{I}}\}$ denotes that of inequivalent induced $G$-modules
${\rm Ind}(\phi_i)=:\hat\phi_i$.

Let $V$ be the natural finite-dimensional $G$-module, denote the $N$-restriction
of $V$ by $\check{V}$.
The following tensor products decompose into irreducible components: 
\begin{equation*}\label{1.2}
  \wedge^{r} \check{V}\otimes \check\rho_j=\bigoplus\limits_{i\in\check{\rm I}}b^{(r)}_{ji} \check\rho_{i} \qquad {\rm and}\qquad
  \wedge^{r} V\otimes \hat\phi_j=\bigoplus\limits_{i\in\hat{\rm I}}d^{(r)}_{ji} \hat\phi_i,
\end{equation*}
where the integral matrices $B_r=(b^{(r)}_{ij})$ and $D_r=(d^{(r)}_{ij})$ are of the same size for $1\leq r\leq n-1$ respectively.
For $r=1$,  we can construct the representation graph $\mathcal{R}_{V}(\check{G})$
(resp. $\mathcal{R}_{V}({\hat{N}})$) by taking the elements of ${\rm \check{I}}$ (resp. ${\rm \hat{I}}$) as vertices,
connecting $i$ and $j$ with $\mathrm{max}(b^{(1)}_{ij}, b^{(1)}_{ji})$ (resp. $\mathrm{max}(d^{(1)}_{ij}, d^{(1)}_{ji})$) edges,
and adding an arrow pointing to $i$ if $b^{(1)}_{ij}>1$ (resp. $d^{(1)}_{ij}>1$). The digraph $\mathcal{R}_{V}(\check{G})$ (resp. $\mathcal{R}_{V}({\hat{N}})$) is called
a \textit{generalized McKay-Slodowy quiver}.

If $N=G\leq \mathrm{SL}_2(\mathbb{C})$,
the tensor product between $G$-module $\wedge^{r} V$ and an
irreducible $G$-module $\rho_j$ ($j\in {\rm I}_{G}$) come down to
\begin{equation*}\label{1.4}
\wedge^{r} V \bigotimes \rho_j = \bigoplus_{i\in {\rm I}_{G}} a^{(r)}_{ji} \rho_i,
\end{equation*}
where the matrices $A_r=(a^{(r)}_{ij})$ for $1\leq r\leq n-1$. 
Thus, the generalized McKay quiver $\mathcal{R}_{V}(G)$
defined in \cite{HJC} is just $A_r$ for $r=1$, whose index set is ${\rm I}_{G}$, and there are
$a^{(1)}_{ij}$ directed edges from $i$ to $j$, and an undirected edge between $i$ and $j$ represents the pair of arrows between $i$ and $j$.

Let $N\unlhd G$ be a certain pair of subgroups of $\mathrm{SL}_2(\mathbb{C})$,
the McKay-Slodowy quiver $\mathcal{R}_{V}(\check{G})$ (resp. $\mathcal{R}_{V}({\hat{N}})$) is
the twisted (resp. the non-twisted) multiply laced affine Dynkin diagram,
a detailed description of the McKay-Slodowy correspondence is available in \cite{JWZ}.
When $N=G\leq\mathrm{SL}_2(\mathbb{C})$, the McKay-Slodowy correspondence descends to the McKay correspondence, where
the McKay quiver
$\mathcal{R}_{V}(G)$ is the simple-laced affine Dynkin diagram. Furthermore,
a generalized McKay quiver $\mathcal{R}_{V}(G)$ had been described in \cite{HJC}
for $G$ is a finite subgroup of $\mathrm{SL}_3(\mathbb{C})$.
An impetus to seek a generalized McKay-Slodowy correspondence is that
the general formulas of Poincar\'{e} series for the relative invariants
might be helpful for general minimal resolutions of the singularity in higher dimension.

After giving explicit formulas for the Poincar\' series for relative symmetric invariants for the $\mathrm{SL}_n(\mathbb C)$
in terms of the intrinsic data of the subgroups, we also consider the relation between the Poincar\'e series of relative symmetric invariants and
the finite and affine Coxeter transformation. Tensor invariants
and Poincar\'e series were used by Benkart \cite{Ben} to realize the exponents of simply laced affine Lie algebras
and we have generalized the realization of exponents to all twisted affine Lie algebras
except $A_{2n}^{(1)}$ in
\cite{JWZ}. Similarly in this paper, we will show that the Poincar\'e series of symmetric invariants also have very close relation with
finite and affine Coxeter transformations, and we also generalize some of these formulae of symmetric invariants to all twisted affine and
non-simply laced untwisted affine Lie algebras. Not surprisingly, we are also able to generalize Kostant's compact formulas of the
Poincar\'e series to all untwisted and twisted affine Dynkin diagrams using intrinsic group data in the context of the McKay-Slodowy correspondence.

We also derive the Poincar\'e series of symmetric invariants exclusively in terms of Tchebychev polynomials. This implies a surprising beautiful
fact about the Poincar\'e series of invariants that they are completely determined by the types of the distinguished pairs of subgroups and
the respective dimensions of the subgroups. Our new formula points out the global picture of the Poincar\'e series of invariants, for example, one does
not need to know the information of exponents or eigenvalues of the Coxeter transformation to determine the Poincar\'e series.
So far this information has only been implicitly available for general symmetric invariants in the literature.

The paper is organized as follows. In section $2$, we give the
formulas of the Poincar\'{e} series of the $N$-restriction modules and induced $G$-modules
in symmetric algebra $S(\mathbb{C}^n)=\bigoplus\limits_{k \geq 0}S^{k}(\mathbb{C}^n)$ respectively.
In particular, if $N=G\leq\mathrm{SL}_n(\mathbb{C})$, we get
a formula of the Poincar\'{e} series for irreducible $G$-modules in symmetric algebra $S(\mathbb{C}^n)$.
In Section $3$, the Poincar\'{e} series of invariants for pairs of finite subgroups (resp. the finite subgroups) of $\mathrm{SL}_2(\mathbb{C})$
are obtained in terms of the quantum affine (resp. finite) Cartan matrices of the affine (resp. finite) Lie algebras.
Moreover, the Poincar\'{e} polynomial of symmetric invariants exhibit the exponents and Coxeter number of the affine (resp. finite) Lie algebras.
In addition, we generalize a classical result of Poincar\'{e} series for symmetric $G$-invariants to Poincar\'{e} series
for $G$-restriction invariants or $N$-induction invariants. In other words, we have provided an unified formula of Poincar\'{e} series for symmetric invariants for affine Lie algebras in both untwisted and twisted types.
In Section $4$, we get the closed-form expressions of the Poincar\'e series
of symmetric invariants for pairs of subgroups which
realize all twisted and untwisted affine Lie algebras. Moreover, the new formulas are of global nature in that they are completely determined by
the pairs of subgroups.

\section{Poincar\'{e} series associated with $\mathrm{SL}_n(\mathbb{C})$}

Let $N, G$ be a pair of finite subgroups of $\mathrm{SL}_n(\mathbb{C})$ such that $N \unlhd G$.
Let $V=\mathbb C^n$ be the standard $\mathrm{SL}_n$-module, also referred as
the natural $G$-module. Assume that $\{\rho_i\}$ and $\{\phi_i\}$ are the sets of complex irreducible modules for
$N$ and $G$ respectively. As we have remarked that the span of  $N$-restriction modules
$\{\check{\rho}_i\}$ and the span of $G$-induction modules $\hat{\phi}_i$ have the same dimension, so $|{\rm \check{I}}|=|{\rm \hat{I}}|$.

Let $\check{s}^{j}_{k}$ $(j\in {\rm \check{I}})$ (resp. $\hat{s}^{j}_{k}$ ($j\in {\rm \hat{I}}$)) be the multiplicity
of the $N$-restriction $\check{\rho}_j$ in the $k$th symmetric power $S^{k}(V)$ (resp. induced $G$-module $\hat{\phi}_j$ in $S^{k}(V)$),
namely
\begin{align*}
\check{s}^{j}_{k}={{\dim}}\left({Hom}_{N}(\check{\rho}_j, S^{k}(V))\right) \
{\rm and} \ \ \hat{s}^{j}_{k}={{\dim}}\left({Hom}_{G}(\hat{\phi}_j, S^{k}(V))\right).
\end{align*}
Let
\begin{equation*}\label{1.3}
  \check{s}^{j}(t)=\sum\limits_{k\geq 0}\check{s}^{j}_{k}t^{k}\qquad {\rm and}\qquad
  \hat{s}^{j}(t)=\sum\limits_{k\geq 0}\hat{s}^{j}_{k}t^{k}
\end{equation*}
be the Poincar\'{e} series for the multiplicities of $\check\rho_j$ and $\hat\phi_j$ in the symmetric algebra
$S(V)=\bigoplus\limits_{k \geq 0}S^{k}(V)$ respectively.

When $N=G$, we only consider $G$-modules, and let $s^{j}_{k}$ be the multiplicity of an irreducible
$G$-module $\rho_j$ in the $k$th symmetric power $S^{k}(V)$ for each $j\in {\rm I_{G}}$. Accordingly the Poincar\'e series is then
\begin{equation*}\label{mul1}
  s^{j}(t)=\sum\limits_{k\geq 0}s^{j}_{k}t^{k}=\sum\limits_{k\geq 0}{{\dim}}\left({Hom}_{G}(\rho_j, S^{k}(V))\right)t^k.
\end{equation*}

In this section, we will give the general formulas of Poincar\'{e} series for arbitrary pair
$N\unlhd G\leq \mathrm{SL}_n(\mathbb{C})$ in the symmetric algebra $S(\mathbb{C}^n)=\bigoplus\limits_{k \geq 0}S^{k}(\mathbb{C}^n)$.
Moreover, $N=G\leq \mathrm{SL}_n(\mathbb{C})$ leads to a general formula of Poincar\'{e}
series of $G$ in $S(\mathbb{C}^n)=\bigoplus\limits_{k \geq 0}S^{k}(\mathbb{C}^n)$.

\subsection{The general formulas of Poincar\'{e} series for $\mathrm{SL}_n(\mathbb{C})$}

Let $G$ be a finite subgroup of $\mathrm{SL}_n(\mathbb{C})$ and $V=\mathbb C^n$ the standard $\mathrm{SL}_n$-module, which is
also viewed as the
natural $G$-module as $G$-restriction.
We will give explicit expression of $S^k(V)$  and establish its relation with
character values and adjacency matrices.
We first recall some basic results on $\mathrm{SL}_n$-modules.

\begin{lemm} Pieri Rule {\rm\cite[Prop. 15.25]{FuHa}}\label{sym&ext}
Let $\Gamma_{(u_1,\ldots,u_{n-1})}$ be an irreducible representation of $\mathrm{SL}_n(\mathbb{C})$.
Then,
the tensor product of $\Gamma_{(u_1,\ldots,u_{n-1})}$ with $S^{k}(V)=\Gamma_{(k,0,\ldots,0)}$
decomposes into a direct sum :
\begin{equation*}
    \Gamma_{(u_1,\ldots,u_{n-1})}\bigotimes \Gamma_{(k,0,\ldots,0)} = \bigoplus\Gamma_{(b_1,\ldots,b_{n-1})}
\end{equation*}
where the sum is over all $(b_1, \ldots,  b_{n-1})$ for which there are nonnegative integers $c_1, \ldots,  c_{n}$
whose sum is $k$, such that $c_{i+1}\leq u_i$ for $1\leq i \leq n-1$ and $b_i=u_i + c_i -c_{i+1}$ for $1\leq i \leq n-1$.
\end{lemm}

\begin{lemm} {\rm \cite[Lem. 3.4]{JWZ}} \label{reinnumber}
 Let $N$ be a normal subgroup of the finite group $G$,
and $\{\rho_i|i\in {\rm I}_G\}$ (resp. $\{\phi_i|i\in {{\rm I}_N}\}$)
the set of pairwise inequivalent complex irreducible modules of $ G$ (resp. $N$).
Let $\{\check\rho_i|i\in {\rm\check{I}}\}$ be the set of mutually inequivalent $N$-restrictions of $\rho_i's$ such that
$\check\rho_i \cap \check\rho_j=0$ for $i,j\in{\rm \check{I}}$, and
$\{\hat\phi_i|i\in {\rm \hat{I}}\}$ the set of inequivalent induced $ G$-modules.
Then $|\{\check\rho_i|i\in {\rm \check{I}}\}|=|\{\hat\phi_i|i\in {\rm \hat{I}}\}|$,
and the common cardinality is equal to $|\Upsilon(N)|$, where $\Upsilon(N)=\Upsilon\cap N$ and
$\Upsilon$ is a fixed set of $G$-conjugacy class representatives.
\end{lemm}

\begin{lemm}\label{eqsym}
Let $V$ be the standard $\mathrm{SL}_n(\mathbb{C})$-module. Then
for each $k$ the following relations between symmetric and exterior powers of $V$ hold:
\begin{align}\label{sym}
    & \bigoplus_{r=0}^{[n/2]}\wedge^{2r}V\bigotimes S^{k-2r}(V) \nonumber \\
  = & \bigoplus_{r=0}^{[n/2]}\wedge^{2r+1}V\bigotimes S^{k-2r-1}(V),
\end{align}
where $S^0(V)=\wedge^0(V)=\mathbb C$, $\wedge^{r}V=0$ for $r>n$ and $S^r(V)=0$ for $r<0$.
\end{lemm}
\begin{proof} Note that $\dim V=n$, it follows from the Pieri rule ($\mathrm{Lemma}$ \ref{sym&ext}) that
\begin{align*}
   V & \bigotimes S^{k-1}(V) = S^k(V) \bigoplus \Gamma_{(k-2,1,0,\ldots,0)}, \\
   \wedge^2V &\bigotimes S^{k-2}(V) = \Gamma_{(k-2,1,0,\ldots,0)} \bigoplus \Gamma_{(k-3,0,1,0,\ldots,0)}, \\
     & \ \ \ \vdots\\
    \wedge^{n-2}V &\bigotimes S^{k-(n-2)}(V) = \Gamma_{(k-(n-2),0,\ldots,0,1,0)} \bigoplus \Gamma_{(k-(n-1),0,\ldots,0,1)}, \\
    \wedge^{n-1}V & \bigotimes S^{k-(n-1)}(V) = \Gamma_{(k-(n-1),0,\ldots,0,1)}\bigoplus S^{k-n}(V).
\end{align*}
Then in the Grothendieck ring, one sees that
\eqref{sym} holds at the character level, which then implies that the module relations also hold.
\end{proof}

In other words, in the Grothendieck ring of $\mathrm{SL}_n$-modules one has that
\begin{equation}\label{symchar}
[S^k(V)]=\sum_{r=1}^n(-1)^r[\wedge^rV][S^{k-r}(V)].
\end{equation}

Let $N$ be an arbitrary normal group of the finite group $G$. The decomposition of
the tensor product of $V$ and arbitrary induced (resp. restriction) module of an irreducible $N$-module  (resp. $N$-module)
give rise to two adjacency matrices. In \cite{JWZ} we have shown that there
is a deep relation between the character values of the group and
these adjacency matrices in connection with the McKay-Slodowy correspondence, which
generalized Steinberg's results \cite{Ste} for the McKay correspondence. In the following we shall generalize our version of
the McKay-Slodowy correspondence to all exterior powers of the standard $\mathrm{SL}_n(\mathbb{C})$-module.

Here is our group theoretical description of the generalized McKay-Slodowy correspondence
which are direct consequence of Lemma \ref{reinnumber}, Lemmas \ref{tra&char1} and Lemma \ref{tra&char}.


\begin{lemm} \label{tra&char1}
Let $N\unlhd G$ be a pair of finite normal subgroups of $\mathrm{SL}_n(\mathbb{C})$
with $\{\check\rho_i|i\in \rm\check{{I}}\}$
(resp. $\{\hat\phi_i|i\in \rm\hat{{I}}\}$) be the set of $N$-restriction modules
(resp. induced $G$-modules).
Assume matrix $B_r$ (resp. $D_r$) is afforded by the tensor product of $\wedge^{r} \check{V}$
and $\check\rho_i$ (resp. $\hat\phi_i$) for $i\in\check{\rm I}$ (resp. $i\in\hat{\rm I}$)
and $1\leq r\leq n-1$, where $V$ is a natural module.
Let $\chi_{\wedge^{r} V}$ and $\chi_{\check\rho_i}$ (resp. $\chi_{\hat\phi_i}$) be characters of $\wedge^{r} V$
and $\check\rho_i$ (resp. $\hat\phi_i$) respectively. Then,
\begin{enumerate}
  \item The column vectors $(\chi_{\check\rho_i}(g))_{i\in {\rm\check{I}}}$ and $(\chi_{\hat\phi_i}(g))_{i\in {\rm\hat{I}}}$
  are the eigenvectors of the matrices $\chi_{\wedge^{r} V}(1) I-B_r^{T}$ and $\chi_{\wedge^{r} V}(1) I-D_r^{T}$ with eigenvalue
  $\chi_{\wedge^{r} V}(1)-\chi_{\wedge^{r} V}(g)$ respectively  for $1\leq r \leq n-1$,
  where $g$ runs through $\Upsilon(N)=N\cap \Upsilon$, $\Upsilon$ is a set of representatives of conjugacy class of $G$.
  \item The column vectors $(\chi_{\check\rho_i}(1))_{i\in {\rm\check{I}}}$ and $(\chi_{\hat\phi_i}(1))_{i\in {\rm\hat{I}}}$
  are eigenvectors of the matrices $\chi_{\wedge^{r} V}(1) I-B_r^{T}$ and $\chi_{\wedge^{r} V}(1) I-D_r^{T}$  with eigenvalue $0$ respectively.
\end{enumerate}
\end{lemm}

\begin{lemm}\label{tra&char}
Let $G$ be a finite subgroup of $\mathrm{SL}_n(\mathbb{C})$.
Let $V$ be the standard $G$-module and $\{\rho_i|i\in {\rm I}_{G}\}$ be the set of irreducible $G$-modules.
Assume $A_r=(a_{ij}^{(r)})$ is afforded by the tensor product of $\wedge^{r} V $ and $\rho_j$ for $i\in {\rm I}_{G}$ and
 $1\leq r\leq n-1$. Let $\chi_{\wedge^{r} V}$ and $\chi_{\rho_i}$ be characters of $\wedge^{r} V$
and $\rho_i$ respectively. Then
\begin{enumerate}
  \item $A_r=A_{n-r}^{T}$ for $1\leq r \leq n-1$.
  \item The column vector $(\chi_{\rho_i}(g))_{i\in {\rm I}_{G}}$ of the character table of $G$ is an eigenvector of
  $\chi_{\wedge^{r} V}(1) I-{A_{n-r}}$ with eigenvalue $\chi_{\wedge^{r} V}(1)-\chi_{\wedge^{n-r} V}(g)$ for $1\leq r \leq n-1$,
  where $g$ runs over the set $\Upsilon$.
  \item The column vector $(\chi_{\rho_i}(1))_{i\in {\rm I}_{G}}$ of the degrees of character of $G$
  is an eigenvector of matrix $\chi_{\wedge^{r} V}(1) I-{A_{n-r}}$ for $1\leq r\leq n-1$ with eigenvalue $0$.
\end{enumerate}
\end{lemm}
\begin{proof} The last two statements are immediate, so we only prove the first one.
Note that $\wedge^{r} V \bigotimes \rho_j = \bigoplus_{i\in {\rm I}_{G}} a^{(r)}_{ij} \rho_i$, then
\begin{eqnarray*}
  a^{(r)}_{ij} &=& {{\dim}}(\wedge^{r} V \bigotimes \rho_j, \rho_i) \\
   &=& {{\dim}}(\rho_j, (\wedge^{r} V)^* \bigotimes \rho_i) \\
   &=& {{\dim}}(\rho_j, \wedge^{n-r} V \bigotimes \rho_i)=a^{(n-r)}_{ji},
\end{eqnarray*}
which completes the proof.
\end{proof}

For a pair of subgroups $N\unlhd G$ of $\mathrm{SL}_n(\mathbb{C})$ and the standard module $V$,
the formulas of the Poincar\'{e} series for the multiplicities of each $ G$-restriction module or induced $ N$-module
in the symmetric algebra $S(V)=\bigoplus\limits_{k \geq 0}S^{k}(V)$ are given in the next two theorems respectively.

\begin{theo}\label{thm1}
Let $N\unlhd G$ be a pair of finite subgroups of $\mathrm{SL}_n(\mathbb{C})$, and
$V$ the standard $N$-module. Let
$B_r$ be the matrix given by the tensor product of $\wedge^{r} V$ and $\check\rho_i$ ($1\leq r\leq n-1$),
and $M_1^{i}$ the matrix
$(1+(-1)^nt^n)I+\sum\limits_{r=1}^{n-1}(-1)^r B_{n-r} t^r$ with the $i$th column replaced by
$\underline{\delta}=(1, 0, \ldots, 0)^T\in\mathbb N^{|\rm{\check I}|}$.
Then the  Poincar\'e series of the multiplicities of $\check\rho_i$ in $S(V)=\bigoplus\limits_{k \geq 0}S^{k}(V)$ is
\begin{eqnarray}\label{ps1}
   \check{s}^{i}(t)&=& \frac{\mathrm{det}(M_1^{i})}{\mathrm{det}\left((1+(-1)^nt^n)I+\sum\limits_{r=1}^{n-1}(-1)^r B_{n-r}  t^r\right)} \nonumber\\
    &=& \frac{{\mathrm{det}}(M_1^{i})}{\prod\limits_{g\in \Upsilon(N)}\left(1+\sum\limits_{r=1}^{n-1}(-1)^r \chi_{\wedge^{n-r} V}(g) t^r +(-1)^nt^n\right)},
\end{eqnarray}
where $\Upsilon(N)=N\cap \Upsilon$, $\Upsilon$ is a set of representatives of the conjugacy classes of $G$.
\end{theo}

\begin{proof} By the equality \eqref{symchar} the multiplicity of $N$-restriction $\check{\rho}_i$ in $S^{k}(V)$ is
\begin{align*}\label{multiequ}
  \check{s}^{i}_{k} & = {{\dim}}\left({Hom}_{N}(\check\rho_i, S^{k}(V))\right) \nonumber\\
   & =\sum\limits_{r=1}^n (-1)^{r-1}{{\dim}}\left({Hom}_{N}\left(\check\rho_i, \sum\limits_{r=1}^n \wedge^rV\bigotimes S^{k-r}(V)\right)\right).
\end{align*}
Then the Poincar\'e series $\check{s}^{i}(t)$ is computed as follows, 
\begin{eqnarray*}
  \check{s}^{i}(t) &=  
   & \sum\limits_{k\geq 0}{{\dim}}\left({Hom}_{N}(\check\rho_i, S^{k}(V))\right)t^k \nonumber\\
   &=& \delta_{i0}+ \sum\limits_{k\geq 1} \sum\limits_{r=1}^n (-1)^{r-1}{{\dim}}\left({Hom}_{N}\left(\check\rho_i, \sum\limits_{r=1}^n \wedge^rV\bigotimes S^{k-r}(V)\right)\right)t^{k} \nonumber\\
   &=& \delta_{i0}+ \sum\limits_{r=1}^n (-1)^{r-1}\left(\sum\limits_{k\geq 1} {{\dim}}\left({Hom}_{N}
   ((\wedge^rV)^{*}\bigotimes\check\rho_i , S^{k-r}(V))\right)t^{k} \right) \nonumber\\
   &=& \delta_{i0}+ \sum\limits_{r=1}^n (-1)^{r-1}\left(\sum\limits_{k\geq 1} {{\dim}}\left({Hom}_{N}
   (\wedge^{n-r}V\bigotimes\check\rho_i , S^{k-r}(V))\right)t^{k} \right) \nonumber\\
   &=& \delta_{i0}+\sum\limits_{r=1}^{n-1} (-1)^{r-1}\left(\sum\limits_{k\geq 1} {{\dim}}\left({Hom}_{N}
   \left(\sum_{j\in {\check I}} b^{(n-r)}_{ij}\check\rho_j , S^{k-r}(V)\right)\right)t^{k} \right) \nonumber\\
   & & +(-1)^{n-1}\sum\limits_{k\geq 1} {{\dim}}\left({Hom}_{N}(\check\rho_i , S^{k-n}(V))\right)t^{k} \nonumber\\
\end{eqnarray*}
\begin{eqnarray}
     &=& \delta_{i0}+ \sum\limits_{r=1}^{n-1} (-1)^{r-1}\left(\sum_{j\in {\check I}} b^{(n-r)}_{ij}\sum\limits_{k\geq 0} {{\dim}}
   \left({Hom}_{N}(\check\rho_j , S^{k}(V))\right)t^{k} \right)t^r \nonumber\\
    \label{thm1eq1}
   & & +(-1)^{n-1}s^{i}(t)t^n \nonumber\\
   &=& \delta_{i0}+ \sum\limits_{r=1}^{n-1} (-1)^{r-1}\left(\sum_{j\in {\check I}} b^{(n-r)}_{ij} s^{j}(t) \right)t^r +(-1)^{n-1}s^{i}(t)t^n.
\end{eqnarray}
Let $\underline{s}=(\check{s}^i(t))_{i\in {\rm\check{I}}}$ be the column vector formed by the Poincar\'e series,
the identity \eqref{thm1eq1} is then written as the matrix equation
\begin{align*}
 \left((1+(-1)^nt^n)I+\sum\limits_{r=1}^{n-1}(-1)^r B_{n-r} t^r\right)\underline{s}=\underline{\delta}.
\end{align*}
Using the Cramer's rule, the Poincar\'e series $\check{s}^{i}(t)$ is equal to a quotient of two determinants.

Since $\chi_{\wedge^{r} V}(1)-\chi_{\wedge^{r} V}(g)$ ($1\leq r \leq n-1$) are
all eigenvalues of the matrix $\chi_{\wedge^{r} V}(1) I-B_{r}^T$, where
$g$ runs through $\Upsilon(N)=N\cap \Upsilon$ of $\mathrm{Lemma}$ \ref{tra&char1}. Subsequently
$\sum_{r=1}^{n-1}(-1)^{r-1} (\chi_{\wedge^{n-r} V}(1)-\chi_{\wedge^{n-r} V}(g)) t^{r-1}$
are all the eigenvalues of the matrix $\sum_{r=1}^{n-1}(-1)^{r-1}(\chi_{\wedge^{n-r} V}(1) I- B_{n-r}) t^{r-1}$.
Thus,
\begin{align*}
  &\mathrm{det}\left(\lambda I-\left(\sum_{r=1}^{n-1}(-1)^{r-1} B_{n-r} t^{r-1}\right)\right) \\
 =&\prod\limits_{g\in \Upsilon(N)}\left(\lambda-\left(\sum_{r=1}^{n-1}(-1)^{r-1} \chi_{\wedge^{n-r} V}(g) t^{r-1}\right)\right).
\end{align*}
Let $m=|\Upsilon(N)|=|\rm\check{{I}}|$, we have
\begin{align*}
    & \mathrm{det}\left((1+(-1)^nt^n)I+\sum\limits_{r=1}^{n-1}(-1)^r B_{n-r} t^r\right) \\
  = & t^m\mathrm{det}\left((t^{-1}+(-1)^nt^{n-1})I-\left(\sum_{r=1}^{n-1}(-1)^{r-1} B_{n-r} t^{r-1}\right)\right) \\
  = & t^m\prod\limits_{g\in \Upsilon(N)}\left((t^{-1}+(-1)^nt^{n-1})-\left(\sum_{r=1}^{n-1}(-1)^{r-1} \chi_{\wedge^{n-r} V}(g) t^{r-1}\right)\right) \\
  = & \prod\limits_{g\in \Upsilon(N)}\left(1+\sum\limits_{r=1}^{n-1}(-1)^r \chi_{\wedge^{n-r} V}(g) t^r +(-1)^nt^n\right).
\end{align*}
This completes the proof.
\end{proof}

Using the same method, the Poincar\'{e} series of induced $ N$-modules in the symmetric algebra can also be computed:
\begin{theo}\label{thm2}
Let $N\unlhd G$ be a pair of finite normal subgroups of $\mathrm{SL}_n(\mathbb{C})$ and
$\{\hat\phi_i |i\in\rm\hat{{I}}\}$ be the set of the induced
$G$-modules of all irreducible $N$-modules.
Assume that $V$ is the standard $G$-module and
$D_{r}$ is the matrix afforded by $\wedge^{r} V$ tensor $\hat\phi_i$ for $1\leq r\leq n-1$.
Let  $M_2^{i}$ be the matrix
$(1+(-1)^nt^n)I+\sum\limits_{r=1}^{n-1}(-1)^r D_{n-r} t^r$ with the $i$th column replaced by
$\underline{\delta}=(1, 0, \ldots, 0)^T\in\mathbb N^{|{\rm\hat I}|}$.
Then the  Poincar\'e series of $\hat{\phi}_i$ in $S(V)=\bigoplus\limits_{k \geq 0}S^{k}(V)$ is
\begin{eqnarray}\label{ps2}
   \hat{s}^{i}(t)&=& \frac{\mathrm{det}(M_2^{i})}{\mathrm{det}\left((1+(-1)^nt^n)I+\sum\limits_{r=1}^{n-1}(-1)^r D_{n-r} t^r\right)} \nonumber\\
    &=& \frac{{\mathrm{det}}(M_2^{i})}{\prod\limits_{g\in \Upsilon(N)}\left(1+\sum\limits_{r=1}^{n-1}(-1)^r \chi_{\wedge^{n-r} V}(g) t^r +(-1)^nt^n\right)},
\end{eqnarray}
where $\Upsilon(N)=N\cap \Upsilon$, and $\Upsilon$ is a set of representatives of conjugacy classes of $G$.
\end{theo}

If $N=G\leq\mathrm{SL}_n(\mathbb{C})$,
then the method of $\mathrm{Theorem}$ \ref{thm1} also gives
a formula of the Poincar\'{e} series for $G$-modules in $S(V)=\bigoplus\limits_{k \geq 0}S^{k}(V)$
associated with ${\mathrm{Lemma}}$ \ref{eqsym} and ${\mathrm{Lemma}}$ \ref{tra&char}.

\begin{theo}\label{thm3}
Let $G$ be a finite subgroup of $\mathrm{SL}_n(\mathbb{C})$ and $\{\rho_i|i\in {\rm I}_{G}\}$
be the set of complex irreducible $G$-modules. Assume that $V$ is the standard $G$-module and
$A_r$ is the matrix afforded by $\wedge^{r} V$ tensor $\rho_i$ for $1\leq r\leq n-1$.
Let  $M^{i}$ be the matrix
$(1+(-1)^nt^n)I+\sum\limits_{r=1}^{n-1}(-1)^r A_{n-r} t^r$ with the $i$th column replaced by
$\underline{\delta}=(1, 0, \ldots, 0)^T\in\mathbb N^{|{\rm I}_{G}|}$.
Then the  Poincar\'e series of $\rho_i$ in $S(V)=\bigoplus\limits_{k \geq 0}S^{k}(V)$ is
\begin{eqnarray}\label{ps3}
   s^{i}(t) &=& \frac{\mathrm{det}(M^{i})}{\mathrm{det}\left((1+(-1)^nt^n)I+\sum\limits_{r=1}^{n-1}(-1)^r A_{n-r} t^r\right)} \nonumber\\
    &=& \frac{{\mathrm{det}}(M^{i})}{\prod\limits_{g\in \Upsilon}\left(1+\sum\limits_{r=1}^{n-1}(-1)^r \chi_{\wedge^{n-r} V}(g) t^r +(-1)^nt^n\right)},
\end{eqnarray}
where $\Upsilon$ is a set of conjugacy class representatives of $G$.
\end{theo}
\subsection{$\mathrm{SL}_3(\mathbb{C})$ and $\mathrm{SL}_4(\mathbb{C})$}

The classification of the finite subgroups of the special linear groups $\mathrm{SL}_3(\mathbb{C})$ and $\mathrm{SL}_4(\mathbb{C})$ are given in \cite{YaYu, HaHe}.
Let $G$ be a finite subgroup of $\mathrm{SL}_3(\mathbb{C})$ or $\mathrm{SL}_4(\mathbb{C})$ and $V$ the standard $G$-module,
we can see the denominator of \eqref{ps3} are connected only with the character of $V$
by applying ${\mathrm{Lemma}}$ \ref{tra&char} and Theorem \ref{thm3} to $G$.

\begin{coro}
Let $G$ be a finite subgroup of $\mathrm{SL}_3(\mathbb{C})$ and $\{\rho_i|i\in {\rm I}_{G}\}$
the set of complex irreducible $G$-modules. Assume that $V$ is the standard $G$-module and
$A_1$ is the adjacency matrix of $\mathcal{R}_{V}(G)$.
Let  $M^{i}$ be the matrix
$(1-t^3)I- A_1 t+ A_1^T t^2$ with the $i$th column replaced by
$\underline{\delta}=(1, 0, \ldots, 0)^T\in\mathbb N^{|{\rm I}_{G}|}$.
Then the  Poincar\'e series of $\rho_i$ in $S(V)=\bigoplus\limits_{k \geq 0}S^{k}(V)$ is
\begin{eqnarray*}
   s^{i}(t) &=& \frac{\mathrm{det}(M^{i})}{\mathrm{det}\left((1-t^3)I- A_1 t+ A_1^T t^2\right)} \\
    &=& \frac{{\mathrm{det}}(M^{i})}{\prod\limits_{g\in \Upsilon}\left(1-\overline{\chi_{V}(g)} t + \chi_{V}(g) t^2-t^3\right)},
\end{eqnarray*}
where $\overline{\chi_{V}(g)}$ is the conjugate of $\chi_{V}(g)$ and $\Upsilon$ is a set of conjugacy class representatives of $G$.
\end{coro}

\begin{coro}
Let $G$ be a finite subgroup of $\mathrm{SL}_4(\mathbb{C})$ and $\{\rho_i|i\in {\rm I}_{G}\}$
the set of complex irreducible $G$-modules. Assume $V$ is the standard $G$-module,
$A_1$ and $A_2$ are matrices afforded by $V$ and $\wedge^{2} V$ tensored with $\rho_i$ respectively.
Let  $M^{i}$ be the matrix
$(1+t^4)I- A_1 t+A_2 t^2 +A_1^T t^3$ with the $i$th column replaced by
$\underline{\delta}=(1, 0, \ldots, 0)^T\in\mathbb N^{|{\rm I}_{G}|}$.
Then the  Poincar\'e series of $\rho_i$ in $S(V)=\bigoplus\limits_{k \geq 0}S^{k}(V)$ is
\begin{eqnarray*}
   s^{i}(t) &=& \frac{\mathrm{det}(M^{i})}{\mathrm{det}\left((1+t^4)I- A_1^T t+A_2 t^2 -A_1 t^3\right)} \nonumber\\
    &=& \frac{{\mathrm{det}}(M^{i})}{\prod\limits_{g\in \Upsilon}\left(1-\overline{\chi_{V}(g)} t + \frac{1}{2}({\chi_{V}(g)}^2-\chi_{V}(g^2))t^2-\chi_{V}(g) t^3+t^4\right)},
\end{eqnarray*}
where $\overline{\chi_{V}(g)}$ is the conjugate of $\chi_{V}(g)$ and $\Upsilon$ is a set of conjugacy class representatives of $G$.
\end{coro}

\subsection{The McKay-Slodowy correspondence} 

\qquad

It is well known that the isomorphism classes of
the finite subgroups of special linear group $\mathrm{SL}_2(\mathbb{C})$  are
a cyclic group ${C}_n$ of order $n$, a binary dihedral group ${D}_{n}$ of order $4n$ and
three exceptional polyhedral groups: the binary tetrahedral ${T}$ of order $24$,
the binary octahedral group ${O}$ of order $48$, and
the  binary icosahedral group ${I}$ of order $120$.

The McKay-Slodowy correspondence says that the distinguished pairs $N\lhd G$ are $ D_{n-1} \lhd  D_{2(n-1)}$,
$ C_{2n} \lhd  D_n$, $ C_{2n} \lhd  D_{2n}$, $ T \lhd  O$,
$ D_2 \lhd  T$, $ C_2 \lhd  D_2$ and $V\cong\mathbb{C}^2$,
where the corresponding quiver $\mathcal{R}_{V}({\check{G}})$ realizes
the twisted multiply laced affine Dynkin diagram of types
$A_{2n-1}^{(2)}$, $D_{n+1}^{(2)}$, $A_{2n}^{(2)}$, $E_6^{(2)}$, $D_4^{(3)}$, or $A_2^{(2)}$ respectively,
and the quiver $\mathcal{R}_{V}({\hat{N}})$ realizes
the non-twisted multiply laced affine Dynkin diagram of type
$B_n^{(1)}$, $C_n^{(1)}$, $C_n^{(1)}$, $F_4^{(1)}$, $G_2^{(1)}$, or $A_1^{(1)}$ respectively.

Let $X$ be the adjacency matrix of the McKay quiver, 
then $2I-X^T$ is the corresponding affine Cartan matrix, and
 denominators $(1+t^2)I-t X^T$ in \eqref{ps1} or \eqref{ps2} is called the quantum affine Cartan matrix of the same type \cite{Su},
 as the limit of $t\to 1$ is the former. 



With respect to the simply or multiply laced affine Dynkin diagrams and quantum affine matrices,
the Poincar\'{e} series in ${\rm Section}$ $2.1$ can be given
as follows.

\begin{coro}\label{cor1}
Let $ N\lhd  G\leq \mathrm{SL}_2(\mathbb{C})$ and
$\{\check\rho_i|i\in {\rm\check{I}}\}$ be the set of $ N$-restriction modules.
Let $B_{1}$ be the adjacency matrix of $\mathcal{R}_{V}({\check{G}})$ and $M_1^{i}$ be the
transpose of quantum affine matrix $(1+t^2)I-tB_{1}$ with the $i$th column replaced by
$\underline{\delta}=(1, 0, \ldots, 0)^T\in\mathbb N^{|{\rm \check I}|}$.
Then the  Poincar\'e series of $\check\rho_i$ in $S(V)=\bigoplus\limits_{k \geq 0}S^{k}(V)$ is
\begin{equation*}
   \check{s}^{i}(t)=\frac{\mathrm{det}(M_1^{i})}{\mathrm{det}((1+t^2)I-tB_{1})}
    =\frac{{\mathrm{det}}(M_1^{i})}{\prod\limits_{g\in \Upsilon(N)}(1+t^2 - \chi_{V}(g)t)},
\end{equation*}
where $\Upsilon(N)=N\cap \Upsilon$,
$\Upsilon=\{g\in G|$ $g$ is a set of representatives of conjugacy class of $G\}$.
\end{coro}

\begin{coro}\label{cor2}
Let $ N\lhd  G\leq \mathrm{SL}_2(\mathbb{C})$ and
$\hat\phi_i$ $(i\in\hat{\mathrm{I}})$ be an induced $ G$-module.
Let $D_{1}$ be the adjacency matrix of $\mathcal{R}_{V}({\hat{N}})$ and $M_2^{i}$ the transpose of
quantum affine matrix $(1+t^2)I-t D_{1}$ with the $i$th column replaced by $\underline{\delta}=(1, 0, \ldots, 0)^T\in\mathbb N^{|{\rm\hat I}|}$.
Then the  Poincar\'e series of $\hat{\phi}_i$ in $S(V)=\bigoplus\limits_{k \geq 0}S^{k}(V)$ is given by
\begin{equation*}
   \hat{s}^{i}(t)=\frac{{\mathrm{det}}(M_2^{i})}{\mathrm{det}((1+t^2)I-tD_{1})}
   =\frac{{\mathrm{det}}(M_2^{i})}{\prod\limits_{g\in \Upsilon(N)}(1+t^2-\chi_{V}(g)t)},
\end{equation*}
where $\Upsilon(N)=N\cap \Upsilon$,
$\Upsilon$ is a set of conjugacy class representatives of $ G$.
\end{coro}

\begin{coro}\label{cor3}
Let $G\leq\mathrm{SL}_2(\mathbb{C})$ and $\{\rho_i|i\in {\rm I}_{G}\}$
be the set of complex irreducible $ G$-modules. Let $A_1$ be the adjacency matrix of
$\mathcal{R}_{V}(G)$ and $M^{i}$ be the quantum affine  matrix $(1+t^2)I- t A_1$
with the $i$th column replaced by $\underline{\delta}=(1, 0, \ldots, 0)^T\in\mathbb N^{|{\rm I}_{G}|}$.
Then the  Poincar\'e series of $\rho_i$ in $S(V)=\bigoplus\limits_{k \geq 0}S^{k}(V)$ is
\begin{eqnarray*}
   s^{i}(t) = \frac{\mathrm{det}(M^{i})}{\mathrm{det}\left((1+t^2)I-t A_1 \right)}
    =\frac{{\mathrm{det}}(M^{i})}{\prod\limits_{g\in \Upsilon}\left(1+t^2-\chi_{V}(g) t \right)},
\end{eqnarray*}
where $\Upsilon$ is a set of conjugacy class representatives of $ G$.
\end{coro}

Thanks to the close correlation between twisted and non-twisted multiply laced affine Lie algebras,
the Poincar\'{e} series of affine Lie algebras have the following
relations, which is analogous to our results on Poincar\'{e} series for
restriction and induced modules \cite{JWZ}.

\begin{coro}\label{relation}
Let $N\lhd G$ be $ D_{n-1} \lhd  D_{2(n-1)}$,
$ C_{2n} \lhd  D_n$,  $ T \lhd  O$,
$ D_2 \lhd  T$ in $\mathrm{SL}_2(\mathbb{C})$.
Then the Poincar\'e series $\check{s}^{i}(t)$ and $\hat{s}^{i'}(t)$ for $N$-restriction module $\check\rho_i$
and induced $G$-module $\hat\phi_{i'}$ in $S(\mathbb{C}^2)=\bigoplus\limits_{k \geq 0}S^{k}(\mathbb{C}^2)$ have the following proportional relation:
\begin{align*}
\check{s}^{i}(t)=
\left\{\begin{array}{ll}
  \hat{s}^{i'}(t), & \mbox{$i'$ is a long root in $\mathcal{R}_{V}(\hat{N})$}\\
  |G:N|\hat{s}^{i'}(t),
   & \mbox{$i'$ is a short  root in $\mathcal{R}_{V}(\hat{N})$}
\end{array}\right.,
\end{align*}
where $\check\rho_i\leftrightarrow \hat\phi_{i'}$ is a bijective map between ${\rm\check{I}}$ and ${\rm\hat{I}}$.
In addition, let $N\lhd G$ be $ C_{2n} \lhd  D_{2n}$ $(n\geq 2)$, $ C_2 \lhd  D_2$,
then
\begin{align*}
\check{s}^{i}(t)=
\left\{\begin{array}{ll}
  \hat{s}^{0}(t), & \hbox{\emph{i} is the affine vertex of $\mathcal{R}_{V}(\check{G})$ corresponding to the trivial module}   \\
  2\hat{s}^{i'}(t), & \hbox{\emph{i}\ (resp. $i'$) in the finite Dynkin diagram of $\mathcal{R}_{V}(\check{G})$ (resp. $\mathcal{R}_{V}(\hat{N})$)}
\end{array}\right..
\end{align*}
\end{coro}

\section{Poincar\'e series of symmetric invariants for $\mathrm{SL}_2(\mathbb{C})$}

Let $ N \unlhd  G$ be a pair of finite subgroups of $\mathrm{SL}_2(\mathbb{C})$.
By removing the affine vertex corresponding to the trivial module,
the McKay quivers are the multiply laced Dynkin diagrams corresponding to finite dimensional simple Lie algebras.
Thus we denote the quantum (finite dimensional) Cartan matrices of the non-simply laced Dynkin diagrams
of $\mathcal{R}_V ({\check{G}})$ and $\mathcal{R}_{V}({\hat{N}})$ by $(1+t^2)I-t \widetilde{B}_{1}^{T}$ and $(1+t^2)I-t \widetilde{D}_{1}^{T}$ respectively.

For $ N \lhd  G\leq\mathrm{SL}_2(\mathbb{C})$, the Poincar\'{e} series of
 $ N$-invariants and $ G$-invariants in the symmetric algebra $S(\mathbb{C}^2)$ are
$\check{s}^{0}(t)$ and $\hat{s}^{0}(t)$, then $\check{s}^{0}(t)$ and $\hat{s}^{0}(t)$
are the quotient of quantum finite Cartan matrix by the quantum affine Cartan matrix for non-simply laced Lie algebras.

\begin{theo}\label{inv1}
Let $ N \unlhd  G\leq \mathrm{SL}_2(\mathbb{C})$.
Let $(1+t^2)I-tB_{1}$ (resp. $(1+t^2)I-tD_{1}$) be the transpose of quantum affine Cartan matrix of $\mathcal{R}_V({\check{G}})$
(reps. $\mathcal{R}_{V}({\hat{N}})$), and $(1+t^2)I-t \widetilde{B}_{1}$ (resp. $(1+t^2)I-t \widetilde{D}_{1}$)
the corresponding quantum Cartan matrix.
Then the Poincar\'{e} series of the $ N$-invariants $S(V)^{ N}$ and $ G$-invariants $S(V)^{ G}$
in $S(V)=\bigoplus\limits_{k \geq 0}S^{k}(V)$ are
\begin{align}\label{inv11}
   \check{\textmd{s}}^{0}(t)=\hat{\textmd{s}}^{0}(t)
   &=\frac{\mathrm{det}\left((1+t^2)I-t\widetilde{B}_{1}\right)}{\mathrm{det}((1+t^2)I-tB_{1})}
   =\frac{\mathrm{det}\left((1+t^2)I-t\widetilde{D}_{1}\right)}
   {\mathrm{det}((1+t^2)I-tD_{1})} \nonumber \\
   &=\frac{\mathrm{det}\left((1+t^2)I-t \widetilde{B}_{1}\right)}
   {\prod\limits_{g\in \Upsilon(N)}(1+t^2-\chi_{V}(g)t)},
\end{align}
where $\Upsilon(N)=\Upsilon\cap N$ and $\Upsilon$ is
a fixed set of representatives of conjugacy class of $G$.
\end{theo}

Similarly, we have also the quantum finite Cartan matrix $(1+t^2)I- t \widetilde{A}_1$ for
 a finite simply laced Lie algebra, where $\widetilde{A}_1$ is the adjacency matrix of the simply laced Dynkin diagram
 corresponding to $\mathcal{R}_{V}(G)$ for $N=G\leq\mathrm{SL}_2$.

\begin{theo}\label{inv2}
Let $G\leq\mathrm{SL}_2(\mathbb{C})$.
Let $(1+t^2)I- t A_1$ be the quantum affine Cartan matrix of $\mathcal{R}_{V}(G)$,
and $(1+t^2)I- t \widetilde{A}_1$ be the corresponding quantum finite Cartan matrix.
Then the Poincar\'e series of $G$-invariants $S(V)^{G}$ inside $S(V)=\bigoplus\limits_{k \geq 0}S^{k}(V)$ is
\begin{eqnarray}\label{inv22}
   s^{0}(t) = \frac{\mathrm{det}\left((1+t^2)I- t \widetilde{A}_1 \right)}{\mathrm{det}\left((1+t^2)I- t A_1 \right)}
    =\frac{\mathrm{det}\left((1+t^2)I- t \widetilde{A}_1 \right)}{\prod\limits_{g\in \Upsilon}\left(1+t^2-\chi_{V}(g) t \right)},
\end{eqnarray}
where $\Upsilon=\{g\in G|$ $g$ is a representative of conjugacy class of $G\}$.
\end{theo}


As application of the Poincar\'e series of symmetric invariants, we now consider
their relation with the finite and affine Coxeter transformation. Tensor invariants
and Poincar\'e series have been shown to realize the exponents of simply laced affine Lie algebras
by Benkart \cite{Ben}, and we have generalized her result to all types of affine Dynkin diagrams except $A_{2n}^{(1)}$ in
\cite{JWZ}. In the following, we will study the connection between the Poincar\'e series of symmetric invariants
and show that they also provide a good setting to give rise to all eigenvalues of the
affine and finite Coxeter transformations.

\ \ \ \ \ \ \ \ \ \ \ \ \ \ \ \ \ \
$\textbf{Table 1}$\ \ \ \ \ \ \ \ \ Exponents and Coxeter number.

\noindent\begin{tabular}{l@{}c@{}c@{}}
  \hline
  Dynkin diagrams & Exponents & Coxeter number   \\
  \hline
  $A_n$  & $1,2,3,\ldots,n$ & $n+1$  \\
  $B_n$  & $1,3,5,\ldots,2n-1$ & $2n$  \\
  $C_n$ & $1,3,5,\ldots,2n-1$ & $2n$   \\
  $D_n$ & $1,3,5,\ldots,2n-3,n-1$ & $2n-2$  \\
  $E_6$   & $1,4,5,7,8,11$  & $12$  \\
  $E_7$   &  $1,5,7,9,11,13,17$ &  $18$  \\
  $E_8$   & $1,7,11,13,17,19,23,29$  &  $30$ \\
  $F_4$ & $1,5,7,11$ & $12$   \\
  $G_2$ & $1,5$ & $6$  \\
  $A_{1}^{(1)}$ & $0,1$  & $1$  \\
  $A_{2\ell+1}^{(1)}$   &  $0,1,1,\ldots,\ell,\ell,\ell+1$ &  $\ell+1$ \\
  $D_{2\ell+1}^{(1)}$   & $0,2,\ldots,2\ell-2,2\ell-1,2\ell-1,2\ell,
  \ldots,2(2\ell-1)$  &  $2(2\ell-1)$ \\
  $D_{2\ell}^{(1)}$   &  $0,\ldots,\ell-2,\ell-1,\ell-1,\ell-1,\ell,\ldots,
  2\ell-2$ &  $2\ell-2$ \\
  $E_6^{(1)}$   &  $0,2,2,3,4,4,6$  &  $6$ \\
  $E_7^{(1)}$   &  $0,3,4,6,6,8,9,12$ &  $12$ \\
  $E_8^{(1)}$   &  $0,6,10,12,15,18,20,24,30$ & $30$  \\
  $A_{2}^{(2)}$ & $0,2$  & $2$  \\
  $A_{2\ell}^{(2)}$ & $0,1,\ldots,\ell$  & $\ell$ \\
  $B_{2\ell+1}^{(1)}$, $A_{4\ell+1}^{(2)}$ & $0,1,\ldots,\ell-1,\ell,\ell,\ell+1,\ldots,2\ell$ & $2\ell$  \\
  $B_{2\ell}^{(1)}$, $A_{4\ell-1}^{(2)}$ & $0,2,\ldots,2\ell-2,2\ell-1,2\ell,\ldots,2(2\ell-1)$ & $2(2\ell-1)$  \\
  $C_\ell^{(1)}$, $D_{\ell+1}^{(2)}$ & $0,1,\ldots,\ell$ & $\ell$  \\
  $F_4^{(1)}$, $E_6^{(2)}$ & $0,2,3,4,6$ & $6$  \\
  $G_2^{(1)}$, $D_4^{(3)}$ & $0,1,2$ & $2$  \\
  \hline
\end{tabular}

\ \ \ \ \ \ \ \ \

\begin{theo}\label{affps}
Let {\rm $N \unlhd G\leq\mathrm{SL}_2(\mathbb{C})$} and $(1+t^2)I-tB_{1} $ (resp. $(1+t^2)I-t \widetilde{B}_{1}$)
be the transpose of quantum affine (resp. finite) Cartan matrix of $\mathcal{R}_{V}({\check{G}})$ (resp. corresponding finite Dynkin diagram).
Let $\Delta$ (resp. $\widetilde{\Delta}$) be the set of exponents $m_i$ (resp. $\widetilde{m}_i$) of the affine (resp. finite
simple) Lie algebra  associated with the Dynkin diagram
and $h$ (resp. $\widetilde{h}$) the affine (resp. finite) Coxeter number.
Then the Poincar\'{e} series for the $ N$-invariants and $ G$-invariants in
$S(V)=\bigoplus\limits_{k \geq 0}S^{k}(V)$ is
\begin{align*}
  \check{\textmd{s}}^{0}(t)=\hat{\textmd{s}}^{0}(t)
   &=\frac{\mathrm{det}\left((1+t^2)I-t\widetilde{B}_{1}\right)}{\mathrm{det}((1+t^2)I-tB_{1})}
   =\frac{\mathrm{det}\left((1+t^2)I-t\widetilde{B}_{1}\right)}
   {\prod\limits_{g\in \Upsilon(N)}(1+t^2-\chi_{V}(g)t)} \\
   &=\frac{\prod\limits_{\widetilde{m}_i\in \widetilde{\Delta}}\left(1+t^2-2\cos\left(\frac{\widetilde{m}_i \pi}{\widetilde{h}}\right)t\right)}
   {\prod\limits_{m_i\in \Delta}\left(1+t^2-2\cos\left(\frac{m_i\pi }{h}\right)t\right)},
\end{align*}
where $\Upsilon(N)=N\cap \Upsilon$, and $\Upsilon$ is a set of conjugacy class representatives of $ G$.
\end{theo}

\begin{theo}\label{finps}
Let $G\leq\mathrm{SL}_2(\mathbb{C})$ without $G\cong{C}_n$ for $n$ is odd,
and let $(1+t^2)I- A_1$ (resp. $(1+t^2)I- \widetilde{A}_1$) be quantum affine (resp. finite) Cartan matrix of $\mathcal{R}_{V}(G)$
(resp. corresponding finite Dynkin diagram).
Let $\Delta$ (resp. $\widetilde{\Delta}$) be the set of exponents $m_i$ (resp. $\widetilde{m}_i$)
of the simply laced affine (resp. finite) Lie algebra and $h$ (resp. $\widetilde{h}$) the affine (resp. finite) Coxeter number.
Then the  Poincar\'e series of $G$-invariants $S(V)^{G}$ in $S(V)=\bigoplus\limits_{k \geq 0}S^{k}(V)$ is
\begin{align}
   s^{0}(t) &= \frac{\mathrm{det}\left((1+t^2)I- \widetilde{A}_1 t\right)}{\mathrm{det}\left((1+t^2)I- A_1 t\right)}
    =\frac{\mathrm{det}\left((1+t^2)I- \widetilde{A}_1 t\right)}{\prod\limits_{g\in \Upsilon}\left(1+t^2-\chi_{V}(g) t \right)}  \nonumber\\
    &=\frac{\prod\limits_{\widetilde{m}_i\in \widetilde{\Delta}}\left(1+t^2-2\cos\left(\frac{\widetilde{m}_i \pi}{\widetilde{h}}\right)t\right)}
   {\prod\limits_{m_i\in \Delta}\left(1+t^2-2\cos\left(\frac{m_i\pi }{h}\right)t\right)}, \label{exponent}
\end{align}
where $\Upsilon=\{g\in G|$ $g$ is a representative of conjugacy class of $G\}$.
\end{theo}

\begin{remark}
The proof of the above two theorems are similar to our results in \cite[Theorem 4.5]{JWZ}.
In addition, the conclusions also imply that the character value $\chi_{V}(g)$ when $g$ runs over the set $\Upsilon(N)$
(resp. $\Upsilon$) are equal to $2\cos(\frac{m_i\pi }{h})$ when $m_i$ runs through $\Delta$ the set of exponents of the
non-simply laced (resp. simply laced) affine  Lie algebra and $h$ is the corresponding affine Coxeter number \cite{Ben,JWZ}.
\end{remark}

Let $G$ be a subgroup of $\mathrm{SL}_2(\mathbb{C})$, then $G=\langle x,y,z| x^p=y^q=z^r=xyz\rangle$ is a polyhedral group
with $p\geq q \geq 1$ and $p=q$ if $r=1$. Assume that 
$G$ is not a cyclic group of odd order.
The Poincar\'e series for $\textmd{G}$-invariants ${S}(\mathbb{C}^2)^{G}$
have been studied extensively for the simply laced types by Gonzalez-Sprinberg and Verdier \cite{G-SV}, Kn\"{o}rrer \cite{Kn} and Kostant \cite{Kos2}, who gave a compact formula
of the Poincar\'e formula using intrinsic data of the finite group $G$. We can generalize the result to all twisted affine Dynkin diagrams as follows.

Let $N\lhd G\leq \mathrm{SL}_2$ be a distinguished pair of subgroups. Let $\check{\rho}_i$ be the $N$-restriction module of the irreducible $G$-module ${\rho}_i$,
and $\phi_j$ is an irreducible $N$-module.
We let $h=\sum_{i\in \Upsilon\cap N}{\rm dim}\check{\rho}_i$, which is seen the Coxeter number
of the {\it finite} dimensional Lie algebra associated to the affine Dynkin diagram \cite{JWZ}. Define
\begin{align*}
a&=2{\rm max}\{ {\rm dim} \phi_i|i\in \Upsilon \}\\
b&=h+2-a.
\end{align*}

The following result can be proved similarly as in the simply laced type.


\ \ \ \ \ $\textbf{Table 2}$

\noindent\begin{tabular}{l@{\indent}c@{\indent}c@{\indent}c@{\indent\indent}
c@{\indent}c@{\indent}c@{\indent}}
  \hline
  Dynkin diagrams & $a$ &  $b$ & $h$ & $p$ & $q$ &  $r$ \\
  \hline
  $A_{\ell}^{(1)}$ & $2$  & $\ell+1$  & $\ell+1$ & $\frac{1}{2}(\ell+1)$  & $\frac{1}{2}(\ell+1)$ & $1$  \\
  $D_{\ell}^{(1)}$ & $4$  & $2\ell-4$  & $2\ell-2$ & $\ell-2$ & $2$ & $2$  \\
  $E_{6}^{(1)}$ & $6$ & $8$ & $12$ & $3$ & $3$ & $2$  \\
  $E_7^{(1)}$   & $8$ & $12$ & $18$ & $4$ & $3$ &  $2$  \\
  $E_8^{(1)}$   & $12$ & $20$ & $30$ & $5$ & $3$ & $2$ \\
  \hline
  $A_{1}^{(1)}$, $A_{2}^{(2)}$ & $2$  & $2$  & $2$ & $1$  & $1$ & $1$  \\
  \indent\ \ \ $A_{2\ell}^{(2)}$ & $2$  & $2\ell$   & $2\ell$ & $\ell$ & $1$ & $1$  \\
  $B_{\ell}^{(1)}$, $A_{2\ell-1}^{(2)}$ & $4$ & $2\ell-2$ & $2\ell$ & $\ell-1$ & $2$ & $1$  \\
  $C_\ell^{(1)}$, $D_{\ell+1}^{(2)}$ & $2$ & $2\ell$ & $2\ell$ & $\ell$ & $1$ &  $1$  \\
  $F_4^{(1)}$, $E_6^{(2)}$ & $6$ & $8$ & $12$ & $3$ & $2$ & $1$ \\
  $G_2^{(1)}$, $D_4^{(3)}$ & $4$ & $4$ & $6$ & $2$ & $1$ & $1$   \\
  \hline
\end{tabular}

\ \ \ \ \ \ \ \ \ \ \

\begin{theo}\label{uniform}
Let $N \unlhd G\leq\mathrm{SL}_2(\mathbb{C})$ (resp. $G\leq\mathrm{SL}_2(\mathbb{C})$
without $G\cong{C}_n$ for $n$ is odd)
which can realize a pair of multiply (resp. a simply) laced affine Dynkin diagrams.
The Poincar\'{e} series of invariants $\check{s}^{0}(t)$ and $\hat{s}^{0}(t)$ (resp. $s^{0}(t)$) are
\begin{equation}\label{classical}
  \frac{1+t^h}{(1-t^a)(1-t^b)},
\end{equation}
where $a,b,h$ are defined conceptually above, also explicitly listed in ${\rm Table}$ $2$.
\end{theo}

\begin{theo}\label{denom}
Let $N \unlhd G\leq\mathrm{SL}_2(\mathbb{C})$ (resp. $G\leq\mathrm{SL}_2(\mathbb{C})$
without $G\cong{C}_n$ for $n$ is odd).
Let $X$ be the adjacency matrix of $\mathcal{R}_{V}({\check{G}})$ (resp.
$\mathcal{R}_{V}({\hat{N}})$ or $\mathcal{R}_{V}(G)$). Then the
determinant of quantum affine Cartan matrices are given by
\begin{equation}\label{quan}
  \mathrm{det}((1+t^2)I-t X)=\frac{(1-t^{2p})(1-t^{2q})(1-t^{2r})}{1-t^2},
\end{equation}
where the parameters $p,q,r$ are given above.
\end{theo}

\begin{remark} Using folding symmetry of the Dynkin diagrams,
Suter \cite{Su} previously obtained some cases of \eqref{classical}
for the types $B_{n}^{(1)}$, $C_n^{(1)}$,
$F_4^{(1)}$ and $G_2^{(1)}$. He also obtained \eqref{quan} for the simply laced types.
\end{remark}

\section{The global version of Poincar\'e series of invariants for $\mathrm{SL}_2(\mathbb{C})$}

The close relation between the Cartan matrices of the finite Dynkin diagrams and
the Tchebychev polynomials 
have been studied in \cite{D,Kor}, which was further developed to give closed-form expressions of Poincar\'{e}
series of $G$-invariants in the tensor algebra for $G\leq \mathrm{SU}_2$ \cite{Ben}.

In \cite{JWZ} we built upon their relations to compute closed-form expressions of the Poincar\'{e} series of invariants in the tensor algebra
for all twisted and untwisted Dynkin diagrams. 
In this section, we will generalize the results to symmetric invariants
for all affine Dynkin diagrams. Similar to the simply laced case, our new formula further confirms that the Poincar\'e series associated to the twisted affine Lie algebras have equally beautiful relation with the character theory of pairs of subgroups of $\mathrm{SL}_2(\mathbb C)$, more importantly the new formulas are global 
in the sense that the final result depend on the dimension and the type of the pairs of subgroups.


The Tchebychev polynomials of the first kind ${\rm T}_n(t)$ and the second kind ${\rm U}_n(t)$
are recursively defined (also see \cite{Ri}). For $n\geq1$
\begin{align}\label{e:cheb}
  &{\rm T}_0(t)=1, \ \ \ \ {\rm T}_1(t)=t,  \ \ \ \ \ \ {\rm T}_{n+1}(t)=2t{\rm T}_{n}(t)-{\rm T}_{n-1}(t).
  \\
  &{\rm U}_0(t)=1, \ \ \ \ {\rm U}_1(t)=2t,  \ \ \ \ {\rm U}_{n+1}(t)=2t{\rm U}_{n}(t)-
  {\rm U}_{n-1}(t). \label{cf2}
\end{align}
The polynomials ${\rm T}_n(t)$ and ${\rm U}_n(t)$ can be expressed as the additive closed forms
\begin{align}
  {\rm T}_{n}(t)&= {\sum\limits_{i=0}^{\lfloor n/2\rfloor}}{\binom{n}{2i}}t^{n-2i}(t^2-1)^i
  =t^n{\sum\limits_{i=0}^{\lfloor n/2\rfloor}}{\binom{n}{2i}}(1-t^{-2})^i, \label{cf3} \\
  {\rm U}_{n}(t)&= {\sum\limits_{i=0}^{\lfloor n/2\rfloor}}(-1)^i{\binom{n-i}{i}}(2t)^{n-2i}, \label{cf4}
\end{align}
and their factorization are
\begin{align}
{\rm T}_{n}(t)&= 2^{n-1}\prod\limits_{i=1}^n\left(t-{\cos}\left(\frac{(2i-1)\pi}{2n}\right)\right), \label{cf5} \\
{\rm U}_{n}(t)&=2^{n}\prod\limits_{i=1}^n\left(t-{\rm \cos}\left(\frac{\pi i}{n+1}\right)\right).  \label{cf6}
\end{align}
Moreover, they are related by the following relations
\begin{eqnarray}\label{3.14}
 {\rm T}_{n}(t)&=&{\rm U}_{n}(t)-t{\rm U}_{n-1}(t)  \label{cf7} \\
  {\rm and } \ \ \ \  2{\rm T}_{n}(t)&=&{\rm U}_{n}(t)-{\rm U}_{n-2}(t). \label{cf8}
\end{eqnarray}

\subsection{The pair of subgroups ${D}_{n-1}\lhd {D}_{2(n-1)}$}

For $n\geq 3$, let $ D_{2(n-1)}=\langle x, y|x^{2(n-1)}=y^2=-1, yxy^{-1}=x^{-1}\rangle$ be the binary dihedral group of order $8(n-1)$.
$ D_{2(n-1)}$ is imbedded into $\mathrm{SL}_2(\mathbb{C})$ by
\begin{equation}\label{e:imbedD}
  \rho(x)=\left(
      \begin{array}{cc}
        \theta_{4(n-1)}^{-1} & 0 \\
        0 & \theta_{4(n-1)} \\
      \end{array}
    \right), \ \ \ \
  \rho(y)=\left(
      \begin{array}{cc}
        0 &\sqrt{-1}\\
       \sqrt{-1}& 0 \\
      \end{array}
    \right),
\end{equation}
where $\theta_{4(n-1)}=e^{2\pi \sqrt{-1}/{4(n-1)}}$.
Since $\langle x^2, y\rangle= D_{n-1}\lhd  D_{2(n-1)}$,
the $D_{n-1}$-conjugacy classes representatives are $\Upsilon( D_{n-1})=\{\pm1, x^{2i}\ (i=1,\ldots, n-2),y\}$, and
the corresponding character values are $\chi_{V}(x^{2i})=\theta_{4(n-1)}^{2i}+\theta_{4(n-1)}^{-2i}=2\cos(\pi i/(n-1))$, and
$\chi_{V}(y)=0$.

Using \eqref{inv11} in Thm. \ref{inv1}, we get the determinant of the affine Cartan $t$-matrix
\begin{eqnarray}\label{ch1}
  & &  \mathrm{det}((1+t^2)I-t B_1)=\prod\limits_{g\in \Upsilon( D_{n-1})}(1+t^2-\chi_{V}(g)t) \nonumber\\
  &=& (1+t^2-2t)(1+t^2+2t)(1+t^2)\prod\limits_{i=1}^{n-2}\left(1-2\cos\left(\frac{\pi i}{n-1}\right)t\right) \nonumber\\
  &=& (1-t^2-t^4+t^6)\prod\limits_{i=1}^{n-2}\left(1-2\cos\left(\frac{\pi i}{n-1}\right)t\right).
\end{eqnarray}

It is clear that the twisted affine Dynkin diagram of type $A_{2n-1}^{(2)}$ is realized by
the pair of subgroups $\mathrm{D}_{n-1}\vartriangleleft \mathrm{D}_{2(n-1)}$,
there is a multiply laced finite Dynkin diagram of type $C_{n}$ by removing the affine vertex of $A_{2n-1}^{(2)}$.
So the adjacency matrix of
multiply laced finite Dynkin diagram of type $C_{n}$ is
\begin{align*}
\widetilde{B}_1=\left(
      \begin{array}{ccccc}
        0 & 1 & 0 & \cdots & 0 \\
        1 & 0 & 1 & \cdots & 0 \\
        \vdots & \ddots & \ddots & \ddots & 0 \\
        0 & \cdots & 1 & 0 & 2 \\
        0 & 0 & \cdots & 1 & 0 \\
      \end{array}
    \right).
\end{align*}

Let ${\rm c}_{n-1}(t)=\mathrm{det}((1+t^2)I-t \widetilde{B}_1)$. 
The first two are ${\rm c}_0(t)=1+t^2$, ${\rm c}_1(t)=1+t^4$. By expanding the det of the quantum finite Cartan matrix, we have
the recursive relation
\begin{equation*}
  {\rm c}_{n+1}(t)=(1+t^2){\rm c}_n(t)-t^2{\rm c}_{n-1}(t) \ \ \ \ \ \ \ \ {\rm for} \ \ n\geq1.
\end{equation*}
It follows from the definition of Tchebychev polynomial \eqref{e:cheb} that
\begin{align*}\label{ch2}
  {\rm c}_{n-1}(t)=2t^{n}{\rm T}_{n}\left(\frac{t+t^{-1}}{2}\right),
\end{align*}
where ${\rm T}_{n}(t)$ is the Tchebychev polynomial of the first kind. Therefore
\begin{eqnarray}
  {\rm c}_{n-1}(t) &=& 2t^{n}\left(\frac{t+t^{-1}}{2}\right)^{n}{\sum\limits_{i=0}^{\lfloor {n}/2\rfloor}}
  {\binom{n}{2i}}\left(1-\left(\frac{t+t^{-1}}{2}\right)^{-2}\right)^i \nonumber\\
   &=& 2^{(1-n)}{\sum\limits_{i=0}^{\lfloor {n}/2\rfloor}}
  {\binom{n}{2i}}(1+t^2)^{n-2i}(1-t^{2})^{2i}  \label{ch3} \\
   &=& \prod\limits_{i=1}^{n}\left(1+t^2-2{\cos}\left(\frac{\left(2i-1\right)\pi}{2n}\right)t\right). \label{ch4}
\end{eqnarray}

Returning to \eqref{ch1}, by the Laplace expansion we have that
\begin{eqnarray*}
   & & \mathrm{det}((1+t^2)I-tB_1)= (1+t^2){\rm c}_{n-1}(t)-(t^2+t^4){\rm c}_{n-3}(t) \\
   &=& 2t^{n}(1+t^2)\left({\rm T}_{n}\left(\frac{t+t^{-1}}{2}\right)-{\rm T}_{n-2}\left(\frac{t+t^{-1}}{2}\right)\right).
\end{eqnarray*}
By \eqref{cf7}, \eqref{cf2} and \eqref{cf4}, 
\begin{eqnarray*}
  {\rm T}_{n}(t)-{\rm T}_{n-2}(t) &=&(2t^2-2){\rm U}_{n-2}(t) \\
  &=&(2t^2-2){\sum\limits_{i=0}^{\lfloor (n-2)/2\rfloor}}(-1)^i{\binom{n-2-i}{i}}(2t)^{n-2-2i}.
\end{eqnarray*}
Consequently,
\begin{eqnarray*}
   & & \mathrm{det}((1+t^2)I-tB_1)=\\ 
   &&(1-t^2-t^4+t^6){\sum\limits_{i=0}^{\lfloor (n-2)/2\rfloor}}(-1)^i{\binom{n-2-i}{i}}t^{2i}(1+t^2)^{n-2-2i}.
\end{eqnarray*}

Summarizing the above, we have shown the following result.

\begin{theo}\label{thm3.23} Let $ D_{n-1}\lhd  D_{2(n-1)}\leq \mathrm{SL}_2(\mathbb{C})$.
The Poincar\'e series for $ D_{n-1}$-invariants $S(\mathbb{C}^2)^{ D_{n-1}}$ and $ D_{2(n-1)}$-invariants
$S(\mathbb{C}^2)^{ D_{2(n-1)}}$ inside tensor algebra $S(\mathbb{C}^2)=\bigoplus\limits_{k \geq 0}S^{k}(\mathbb{C}^2)$ is
\begin{eqnarray*}
  \check{s}^{0}(t)=\hat{s}^{0}(t)&=&
  \frac{\prod\limits_{i=1}^{n}\left(1+t^2-2{\cos}\left(\frac{\left(2i-1\right)\pi}{2n}\right)t\right)}
  {(1-t^2-t^4+t^6)\prod\limits_{i=1}^{n-2}\left(1-2\cos\left(\frac{\pi i}{n-1}\right)t\right)} \\
  &=&\frac{2^{(1-n)}{\sum\limits_{i=0}^{\lfloor {n}/2\rfloor}}{\binom{n}{2i}}(1+t^2)^{n-2i}(1-t^{2})^{2i}}
  {(1-t^2-t^4+t^6){\sum\limits_{i=0}^{\lfloor (n-2)/2\rfloor}}(-1)^i{\binom{n-2-i}{i}}t^{2i}(1+t^2)^{n-2-2i}}.
\end{eqnarray*}
\end{theo}

\subsection{The pair of subgroups $ C_{2n}\vartriangleleft  D_{n}$}

For $n\geq 2$, let $ D_n=\langle x, y |x^{2n}=y^2=-1, yxy^{-1}=x^{-1}\rangle$ be the binary dihedral group of order $4n$,
then the cyclic group ${C}_{2n}=\langle x\rangle$ is a normal subgroup with index $2$.
By the natural imbedding \eqref{e:imbedD} of $ D_n$ into $\mathrm{SL_2(\mathbb{C})}$,
the conjugacy set $\Upsilon({C}_{2n})=\{\pm1, x^i(i=1,\ldots, n-1)\}$.
The character values are $\chi_{V}(\pm 1)=\pm2$ and $\chi_{V}(x^i)=2\cos(\pi i/n)$, so
\begin{eqnarray*}
  \mathrm{det}((1+t^2)I-tB_1)
  &=&\prod\limits_{g\in \Upsilon({C}_{2n})}(1+t^2- \chi_{\textmd{V}}(g)t) \\
  &=&(1-t^2)^2\prod\limits_{i=1}^{n-1}\left(1-2\cos\left(\frac{\pi i}{n}\right)t\right).
\end{eqnarray*}

The twisted affine Dynkin diagram $D_{n+1}^{(2)}$ is realized by the pair of subgroups $ C_{2n}\lhd D_{n}$. Removing the affine vertex of $D_{n+1}^{(2)}$
we get the multiply laced finite Dynkin diagram of type $B_{n}$, the dual of finite Dynkin diagram $C_{n}$.
Let ${\rm b}_{n-1}(t):={\mathrm{det}}((1+t^2)I-t\widetilde{B}_1)$ be the determinant of the quantum finite
Cartan matrix of finite Dynkin diagram  $B_{n}$, which
satisfies the additive and multiplicative formulas \eqref{ch3} and \eqref{ch4}.

Using the Laplace expansion, we have that
\begin{eqnarray*}
   & & \mathrm{det}((1+t^2)I-tB_1)=(1+t^2){\rm b}_{n-1}(t)-2t^2{\rm b}_{n-2}(t)  \\
   &=& (t^{n-1}(1+t^2)^2-4t^{n+1}){\rm U}_{n-1}\left(\frac{t+t^{-1}}{2}\right) \\
   &=&(1-t^2)^2{\sum\limits_{i=0}^{\lfloor (n-1)/2\rfloor}}(-1)^i{\binom{n-1-i}{i}}t^{2i}(1+t^2)^{n-1-2i}.
\end{eqnarray*}

Therefore, we have the next consequence.

\begin{theo}
Let $ C_{2n}\lhd  D_{n}\leq \mathrm{SL}_2(\mathbb{C})$. The
Poincar\'e series for $ C_{2n}$-invariants and $ D_{n}$-invariants
in the symmetric algebra $S(\mathbb{C}^2)=\bigoplus\limits_{k \geq 0}S^{k}(\mathbb{C}^2)$ is
\begin{eqnarray*}
  \check{s}^{0}(t)=\hat{s}^{0}(t)
  &=& \frac{\prod\limits_{i=1}^{n}\left(1+t^2-2{\cos}\left(\frac{\left(2i-1\right)\pi}{2n}\right)t\right)}
  {(1-t^2)^2\prod\limits_{i=1}^{n-1}\left(1-2\cos\left(\frac{\pi i}{n}\right)t\right)} \\
  &=& \frac{2^{(1-n)}{\sum\limits_{i=0}^{\lfloor {n}/2\rfloor}}{\binom{n}{2i}}(1+t^2)^{n-2i}(1-t^{2})^{2i}}
  {(1-t^2)^2{\sum\limits_{i=0}^{\lfloor (n-1)/2\rfloor}}(-1)^i{\binom{n-1-i}{i}}t^{2i}(1+t^2)^{n-1-2i}}.
\end{eqnarray*}
\end{theo}

\begin{remark}
It is easy to conclude there have the equally Poincar\'e series of invariants for the two pairs of subgroups
${C}_{2n}\lhd {D}_n$ and ${C}_{2n}\lhd {D}_{2n}$ because the same
two-dimensional ${C}_{2n}$-module.
\end{remark}

\subsection{The pairs  $ T\lhd O$, $ D_2\lhd T$ and $ C_{2}\lhd D_{2}$}

For the three pairs of subgroups $ T\lhd O$, $ D_2\lhd T$ and $ C_{2}\lhd D_{2}$ in $\mathrm{SL}_2(\mathbb{C})$,
$\mathrm{Theorem}$ \ref{inv1} says that the Poincar\'e polynomials $\check{s}^0(t)=\hat{s}^{0}(t)$
are expressed as the quotient of the determinants of the quantum affine and finite Cartan matrices,
which in turns carried the information on the exponents of the affine Lie algebras.

We list the Poincar\'e series of the invariants 
using the parameter $a,b,h$ in
${\rm Table}$ $2$ respectively.

\begin{align*}
  \check{s}^{0}(t)=\hat{s}^{0}(t)=&\frac{1+t^{12}}{(1-t^6)(1-t^8)}=1+t^6+t^8+2t^{12}+t^{14}+t^{16}+2t^{18}+\cdots, \\
   \check{s}^{0}(t)=\hat{s}^{0}(t)=&\frac{1+t^6}{(1-t^4)(1-t^4)} \\
   =&1+2t^4+t^6+3t^8+2t^{10}+4t^{12}+3t^{14}+5t^{16}+4t^{18}+\cdots, \\
   \check{s}^{0}(t)=\hat{s}^{0}(t)=&\frac{1+t^2}{(1-t^2)(1-t^2)}\\
   =&1+3t^2+5t^4+7t^6+9t^8+11t^{10}+13t^{12}+15t^{14}+17t^{16}\\
     &+19t^{18}+\cdots.
\end{align*}

\subsection{The groups for simply laced affine Lie algebras}


For $n\geq 3$, let $ C_n=\langle x| x^n=1\rangle$ be the cyclic group of order $n$.
The map $\rho(x)=diag(\theta_n^{-1}, \theta_n)$ ($\theta_n=e^{2\pi \sqrt{-1}/{n}}$) provides a embedding from $ C_n$ into $\mathrm{SL}_2(\mathbb{C})$.
Denote by $\xi_i $ ($i=0,1,\ldots, n-1$) the $n$ irreducible $ C_n$-modules.
The module $V \simeq\xi_1\oplus\xi_{-1}$ and $\chi_{V}(x^i)=\theta_n^i+\theta_n^{-i}$
indicate the determinant of quantum affine Cartan matrix is
$$\mathrm{det}((1+t^2)I-tA_1)=\prod_{i=1}^{n-1}\left(1+t^2-2\cos\left(\frac{2\pi i}{n}\right)t\right).$$

The simply lace affine Dynkin diagram $A_{n-1}^{(1)}$ is realized by the group $C_n$.
Assume $\widetilde{A}_1$ is the adjacency matrix of finite Dynkin diagram $A_{n-1}$ which
adding a affine note is the affine Dynkin diagram $A_{n-1}^{(1)}$.
Let ${\rm a_{n-1}}(t):=\mathrm{det}((1+t^2)I-t \widetilde{A}_1)$.
Set ${\rm a_0}(t)=1$, ${\rm a_1}(t)=1+t^2$.
Expanding the determinant ${\rm a_{n-1}}(t)$ there has an inductive relation
\begin{equation*}\label{ch5}
  {\rm a_{n+1}}(t)=(1+t^2){\rm a_n}(t)-t^2{\rm a_{n-1}}(t), \ \ \ \ \ \ \ \ {\rm for} \ \ n\geq1.
\end{equation*}
What is more, for all $n\geq 0$, we have
\begin{eqnarray}
  {\rm a_{n}}(t) &=& t^n{\rm U_{n}}\left(\frac{t+t^{-1}}{2}\right) \label{ch6}\\
   &=& \prod\limits_{i=1}^n\left(1+t^2-{\rm 2\cos}\left(\frac{\pi i}{n+1}\right)t\right) \label{ch7} \\
   &=& {\sum\limits_{i=0}^{\lfloor n/2\rfloor}}(-1)^i{\binom{n-i}{i}}t^{2i}(1+t^2)^{n-2i}. \label{ch8}
\end{eqnarray}

Identities \eqref{ch6}, \eqref{cf8} and \eqref{cf3} give rise to
\begin{eqnarray*}
   & & \mathrm{det}((1+t^2)I-t A_1) \nonumber\\
   &=& (1+t^2){\rm a_{n-1}}(t)-2t^2{\rm a_{n-2}}(t)-2t^n \nonumber\\
   &=& t^{n}({\rm U_{n}}\left(\frac{t+t^{-1}}{2}\right)-{\rm U_{n-2}}\left(\frac{t+t^{-1}}{2}\right)-2) \\
   &=& t^{n}({\rm 2T_{n}}\left(\frac{t+t^{-1}}{2}\right)-2) \\
   &=& 2^{1-n}{\sum\limits_{i=0}^{\lfloor n/2\rfloor}}
  {\binom{n}{2i}}(1+t^2)^{n-2i}(1-t^2)^{2i}-2t^n.
\end{eqnarray*}

Subsequently, we have now shown the following result.
\begin{theo}
Let ${C}_n\leq \mathrm{SL}_2$. The Poincar\'e series $s^0(t)$ for
${C}_n$-invariants in $S(\mathbb{C}^2)=\bigoplus\limits_{k \geq 0}S^{k}(\mathbb{C}^2)$ is
\begin{align}\label{5.18}
    \textmd{s}^{0}(t)&=\frac{\prod\limits_{i=1}^{n-1}(1+t^2-2\cos(\frac{\pi i}{n})t)}
  {\prod\limits_{i=0}^{n-1}(1+t^2-2\cos(\frac{2\pi i}{n})t)} \nonumber \\
  &=\frac{\sum\limits_{i=0}^{\lfloor (n-1)/2\rfloor}(-1)^i
  {\binom{n-1-i}{i}}t^{2i}(1+t^2)^{n-1-2i}}{2^{1-n}\sum\limits_{i=0}^{\lfloor n/2\rfloor}{\binom{n}{2i}}
   (1+t^2)^{n-2i}(1-t^2)^{2i}-2t^n}.
\end{align}
\end{theo}

\begin{remark}
The Poincar\'{e} series of symmetric invariants for ${C}_n$ of odd order are not included in Theorems \ref{finps},
\ref{uniform} and \ref{denom}. Equality \eqref{5.18} is a uniform formula of
the Poincar\'{e} series of symmetric ${C}_n$-invariants for arbitrary order.
\end{remark}

For the group ${D}_n$ $(n\geq 2)$, Eqs.  \eqref{ch7} and \eqref{ch8} derive the
determinant of the quantum affine Cartan matrix in \eqref{inv22} for $D_{n+2}^{(1)}$:
\begin{eqnarray*}
   & & \mathrm{det}((1+t^2)I-tA_1) = \prod_{g\in \Upsilon}(1+t^2-\chi_{\textmd{V}}(g)t) \nonumber\\
   &=& (1+t^2-2t)(1+t^2+2t)(1+t^2)^2\prod\limits_{i=1}^{n-1}\left(1+t^2-2\cos\left(\frac{\pi i}{n}\right)t\right)\label{ch9} \\
   &=& (1-t^4)^2\sum\limits_{i=0}^{\lfloor (n-1)/2\rfloor}(-1)^i{\binom{n-1-i}{i}}t^{2i}(1+t^2)^{n-1-2i}.  \label{ch10}
\end{eqnarray*}

The group ${D}_n$ realizes the simply laced affine Dynkin diagram of type $D_{n+2}^{(1)}$.
Set ${\rm d_n}(t)=\mathrm{det}((1+t^2)I-t \widetilde{A}_1)$, where $\widetilde{A}_1$ is the adjacency matrix of finite Dynkin diagram $D_{n+2}$
which is obtained by removing the affine vertex of affine Dynkin diagram  $D_{n+2}^{(1)}$.
Assume ${\rm d_0}(t)=1+2t^2+t^4$ and ${\rm d_1}(t)=1+t^2+t^4+t^6$.
By the Laplace expansion, one has the recursive formula
\begin{equation*}
  {\rm d_{n+1}}(t)=(1+t^2){\rm d_n}(t)-t^2{\rm d_{n-1}}(t), \ \ \ \ \ \ \ \ {\rm for} \ \ n\geq1.
\end{equation*}
In addition, we have
\begin{eqnarray*}
  {\rm d_n}(t)&=&2t^{n+1}(1+t^2){\rm T_{n+1}}\left(\frac{t+t^{-1}}{2}\right) \\
   &=& (1+t^2)\prod\limits_{i=1}^{n+1}\left(1+t^2-2{\cos}\left(\frac{(2i-1)\pi}{2(n+1)}\right)t\right) \\
   &=& 2^{-n}{\sum\limits_{i=0}^{\lfloor {(n+1)}/2\rfloor}}
  {\binom{n+1}{2i}}(1+t^2)^{n+2-2i}(1-t^{2})^i. \\
\end{eqnarray*}

Therefore we have that
\begin{theo}
Let ${D}_n\leq\mathrm{SL}_2(\mathbb{C})$. The Poincar\'e series $s^0(t)$ for ${D}_n$ invariants
 in $S(\mathbb{C}^2)=\bigoplus\limits_{k \geq 0}S^{k}(\mathbb{C}^2)$ is given by
{\rm \begin{eqnarray*}
    s^{0}(t) &=& \frac{(1+t^2)\prod\limits_{i=1}^{n+1}\left(1+t^2-2{\cos}\left(\frac{(2i-1)\pi}{2(n+1)}\right)t\right)}
  {(1-t^4)^2\prod\limits_{i=1}^{n-1}\left(1+t^2-{\rm 2\cos}\left(\frac{\pi i}{n}\right)t\right)}  \\
   &=& \frac{2^{-n}{\sum\limits_{i=0}^{\lfloor {(n+1)}/2\rfloor}}{\binom{n+1}{2i}}(1+t^2)^{n+2-2i}(1-t^{2})^{2i}}
   {(1-t^4)^2{\sum\limits_{i=0}^{\lfloor (n-1)/2\rfloor}}(-1)^i{\binom{n-1-i}{i}}t^{2i}(1+t^2)^{n-1-2i}}.
\end{eqnarray*}}
\end{theo}


For the three exceptional polyhedral groups $T$, $O$, $I$ in
$\mathrm{SL}_2(\mathbb{C})$,
we list the Poincar\'e series of symmetric invariants by formula \eqref{exponent} with respect to
the exponents and Coxeter numbers in ${\rm Table}$ $1$ respectively.
\begin{align*}
     s^{0}_T(t)&=\frac{\prod_{\tilde{m}=1,4,5,7,8,11}\left(1+t^2-2\cos(\frac{\pi\tilde{m}}{12})t\right)}
     {\prod_{m=0,2,2,3,4,4,6}\left(1+t^2-2\cos(\frac{\pi m}{6})t\right)}
   =\frac{1+t^2-t^6+t^{10}+t^{12}}{1+t^2-2t^6-2t^8+t^{12}+t^{14}} \\
   &=1+t^6+t^8+2t^{12}+t^{14}+t^{16}+2t^{18}+\cdots,
\end{align*}
\begin{align*}
    s^{0}_O(t)&=\frac{\prod_{\tilde{m}=1,5,7,9,11,13,17}\left(1+t^2-2\cos(\frac{\pi\tilde{m}}{18})t\right)}
     {\prod_{m=0,3,4,6,6,8,9,12}\left(1+t^2-2\cos(\frac{\pi m}{12})t\right)} \\
   &=\frac{1+t^2-t^6-t^8+t^{12}+t^{14}}{1+t^2-t^6-2t^8-t^{10}+t^{14}+t^{16}}=1+t^8+t^{12}+t^{16}+t^{18}+\cdots,
\end{align*}
\begin{align*}
    s^{0}_I(t)&=\frac{\prod_{\tilde{m}=1,7,11,13,17,19,23,29}\left(1+t^2-2\cos(\frac{\pi\tilde{m}}{30})t\right)}
     {\prod_{m=0,6,10,12,15,18,20,24,30}\left(1+t^2-2\cos(\frac{\pi m}{30})t\right)} \\
   &=\frac{1+t^2-t^6-t^8-t^{10}+t^{14}+t^{16}}{1+t^2-t^6-t^8-t^{10}-t^{12}+t^{16}+t^{18}}=1+t^{12}+t^{20}+\cdots. \\
\end{align*}

\vskip30pt \centerline{\bf ACKNOWLEDGMENT}

N. Jing would like to thank the partial support of
Simons Foundation grant 523868 and NSFC grant 11531004. 
H. Zhang would like to thank the support of NSFC grant 11871325.

\end{document}